\documentclass[12pt]{amsart}
\usepackage[latin1]{inputenc}
\usepackage{amssymb,amsmath,amsthm,comment,enumerate,hyperref,multicol,mathdots,tikz}
\usepackage{ytableau}
\usetikzlibrary{cd}

\usepackage[foot]{amsaddr}

\usepackage{thmtools}
\usepackage{thm-restate}

\theoremstyle{definition}
\newtheorem{theorem}{Theorem}[section]
\newtheorem{lemma}[theorem]{Lemma}
\newtheorem{proposition}[theorem]{Proposition}

\definecolor{Green}{RGB}{0,176,80}

\newtheorem{definition}[theorem]{Definition}

\newtheorem{remark}[theorem]{Remark}
\newtheorem{example}[theorem]{Example}

\newtheorem{conjecture}[theorem]{Conjecture}
\newtheorem{collary}[theorem]{Collary}

\usepackage[colorinlistoftodos]{todonotes}

\newcommand{\HH}{\mathcal{H}}

\newcommand{\SDT}{\textrm{SDT}}

\definecolor{darkred}{rgb}{0.7,0,0} 
\newcommand{\defn}[1]{{\color{darkred}\emph{#1}}}

\title{Pieri Rule for GQs Computed via Strict Decomposition Tableaux}

\author[Arroyo]{Joshua Arroyo}
\address{Department of Mathematics,  University of Florida, Gainesville, FL 32611}
\email{joshuaarroyo@ufl.edu}

\begin{document}

\maketitle

\begin{abstract}
    $GQ$ functions are symmetric functions indexed by strict partitions that represent $K$-theoretic Schubert classes in the Lagrangian Grassmannian.
    Buch and Ravikumar \cite{buch2012pieri} proved a Pieri rule for expanding $GQ_{\lambda}\cdot GQ_p$ in terms of $GQ$s via certain shifted skew tableaux.
    In this paper we identify an alternative family of shifted tableaux that enumerates this Pieri rule.
    This partially resolves a conjecture from previous work that these tableaux enumerate the expansion of $GQ_{\lambda} \cdot GQ_{\tau}$ in terms of $GQ$s where $\tau$ is a trapezoid shape.
\end{abstract}

\maketitle

\section{Introduction}

In \cite{ikeda2013k} Ikeda and Naruse introduce $GP$ and $GQ$ functions, which are $K$-theoretic analogues of Schur $P$-- and $Q$--functions.
$GP$s and $GQ$s are representatives of $K$-theoretic Schubert classes of the orthogonal and Lagrangian Grassmannian, respectively.
Both $GP$ and $GQ$ functions are bases for rings and as such products of them can be expanded in terms of their respective functions.
Combinatorial formulas for these expansion are Littlewood-Richardson rules.
For the $GP$ functions a Littlewood-Richardson rule for expanding $GP_{\lambda} \cdot GP_{\mu}$ in terms of $GP$s is known, first proven in \cite{clifford2014ktheoretic}.
This naturally leads to the question of what such a $GQ$ Littlewood-Richardson rule would be.
In \cite{lewis2024combinatorial} Lewis and Marberg proved that $GQ$ functions form the basis of a ring, thus $GQ_{\lambda} \cdot GQ_{\mu}$ can be finitely expanded into $GQ$s.
Currently, the furthest progress to such a rule is a Pieri rule for expanding $GQ_{\lambda}\cdot GQ_p$, done by Buch and Ravikumar in \cite{buch2012pieri}.
The main result of this paper is to give an alternate presentation of the known Pieri rule, computed via a new type of tableaux:
\begin{restatable}{theorem}{arroyoPieri}\label{t:arroyoPieri}
    For a strict partition $\lambda$ and $p\in \mathbb{Z}$,
    \begin{equation*}
        GQ_{\lambda} \cdot GQ_p = \sum_{\mu} |\textrm{SDT}(\mu, \lambda,p)|\; \beta^{|\mu|-|\lambda|-p} GQ_{\mu}
    \end{equation*}
    where $\textrm{SDT}(\mu, \lambda,p)$ is the set of strict decomposition tableaux of shape $\mu$ with reading word equivalent to $w(\lambda, 1, p)$.
\end{restatable}
Strict decomposition tableaux (SDT) were introduced in previous work \cite{arroyo2025typeckstanleysymmetric} as the insertion tableau in a new insertion method for type $C$ Hecke words.
The insertion method conjecturally give combinatorial formulas for Type $C$ $K$-Stanley symmetric functions, $G^C_w$, in terms of $GQ$'s.
In \cite{arroyo2025typeckstanleysymmetric} we showed that for an arbitrary partition $\lambda$ and a trapezoid partition, $\tau$, that $GQ_{\lambda}\cdot GQ_{\tau}$ is a single $G^C_w$.
Therefore, our conjectural formula for Type $C$ $K$-Stanley's also gives a Littlewood-Richardson rule for $GQ$'s when one partition is a trapezoid.

\begin{restatable}{conjecture}{trapezoidConjecture}
\cite{arroyo2025typeckstanleysymmetric}
    \label{conj:trapezoid}
    Let $a \leq b$ positive integers and $\lambda$ a strict shape with $\lambda_1 < k$.
    Then
    \[
GQ_\lambda \cdot GQ_{\tau(a,b)} = \sum_{\nu} |\SDT{(\mu,w(\lambda,a,b))}| \beta^{|\mu|-|\lambda|-ab} GQ_\mu,
    \]
where $\SDT{(\mu,w(\lambda,a,b))}$ is the set of SDT of shape $\mu$ with reading word a $0$-Hecke expression for $w(\lambda,a,b)$.
\end{restatable}
When $a=1$, this formula restricts down to the Pieri rule for $GQ$s.
The fact that this formula agrees with the known Pieri rule despite being described by completely different objects gives evidence that the conjecture is true.

\section{Background}

\subsection{Signed Permutations}
A \defn{signed permutation} $w$ is a permutation of the elements $[\overline{n}]\cup[n]$ such that $w(i)=-w(\overline{i})$.
For example, $v = \overline{1}\overline{3}2\overline{2}31$ is a signed permutation.
By antisymmetry, $w$ is determined by $w([n])$.
When writing a signed permutation, we frequently omit $w([\overline{n}])$; for instance $v = \overline{2}31$.
Signed permutations with composition form a group $W_n$, which is the Coxeter group of Type B/C.
Viewed as a Coxeter group, the generators of $W_n$ are $s_0, s_1, \ldots s_{n-1}$ where $s_0 = (\overline{1},1)$ swaps positions $\overline{1}$ and $1$ and $s_i = (\overline{i{+}1},\overline{i})(i,i{+}1)$ simultaneously swaps positions $i$ and $i+1$ and positions $\overline{i}$ and $\overline{i+1}$ for $i>0$.
These generators satisfy the relations
\begin{enumerate}
    \item Self-inverse: $s_i^2 = 1$ for $i =0,1,\dots,n-1$;
    \item Commutation: $s_i s_j = s_i s_j$ if $|i-j| \geq 2$;
    \item Braid Relation: $s_i s_{i+1} s_i = s_{i+1} s_i s_{i+1}$ if $i > 0$;
    \item Long Braid Relation: $s_1 s_0 s_1 s_0 = s_0 s_1 s_0 s_1$.
\end{enumerate}

The \defn{length} of $w \in W_n$ is the minimum number of generators needed to express $w$, denoted $\ell(w)$.
A \defn{reduced expression} for $w$ is an expression of the form
\[
w = s_{a_1} \cdot s_{a_2} \cdot \ldots \cdot s_{a_p}
\]
where $p = \ell(w)$.
The associated word $(a_1,\dots,a_p)$ is a \defn{reduced word} for $w$.
The Matsumoto--Tits theorem says that the set of reduced words for a given $w$ is connected by the commutation, braid and long braid relations.
For example, the reduced words of $w = \overline{2}31$ are $(1,0,2)$ and $(1,2,0)$.\\

The \defn{0--Hecke monoid} $(W_n,\circ)$ is the monoid on signed permutations obtained by replacing the Self-inverse relation with the Idempotent Relation $s_i \circ s_i = s_i$.
A \defn{0--Hecke expression} for $w \in W_n$ is an expression of the form
\[
w = s_{a_1} \circ s_{a_2} \circ \ldots \circ s_{a_p}
\]
with associated \defn{Hecke word} $(a_1,a_2,\dots,a_p)$.
Let $\HH_p(w)$ be the set of Hecke words for $w$ with $p$ letters and $\HH(w) = \cup_{p \geq 0} \HH_p(w)$.
For example, with $w = \overline{2}31$ we have $\HH_3(w) = \{(1,0,2),(1,2,0)\}$ and
\begin{align*}
\HH_4(w) = \{&(1,0,0,2),(1,0,2,0),(1,0,2,2),(1,1,0,2), (1,2,0,0),(1,2,0,2),(1,2,2,0)\}.
\end{align*}
Note $\HH_{\ell(w)}(w)$ is the set of reduced words for $w$ and $\HH_m(w)$ is empty when $m < \ell(w)$.\\

\noindent For a strict partition $\lambda = (\lambda_1 > \cdots \lambda_k)$, the \defn{Grassmannian signed permutation} of size $n$ and shape $\lambda$ is
    \begin{equation}
        w(\lambda, n) = \overline{\lambda_1} \overline{\lambda_2} \ldots \overline{\lambda_k} i_1 \ldots i_{n-k}
    \end{equation}
    where $n \geq \lambda_1$ and $i_1 < i_2 < \ldots i_{n-k} = [n] \setminus {\lambda_1, \ldots, \lambda_k}$.
    Further, we define the following permutation
    \begin{align*}
        w(\lambda, a, b, n) =\overline{\lambda_1} \overline{\lambda_2} \ldots \overline{\lambda_k}\; i_1 \ldots i_{n-k}\; n{+}a{+}1\; n{+}a{+}2\; \ldots n{+}a{+}b\; n{+}1\; n{+}2\ldots n{+}a 
    \end{align*}
    If the $n$ is omitted in the notation, then take $n = \lambda_1$.
    Notice the following
    \begin{align*}
        &\lambda_{k}{-}1\; \lambda_{k}{-}2 \ldots 0\; \lambda_{k-1}{-}1 \ldots 0 \ldots \lambda_{1}{-}1 \ldots 0\\
        &n{+}a\ldots n{+}a{+}b{-}1\; n{+}a{-}1 \ldots n{+}a{+}b{-}2\ldots n{+}1\ldots n{+}a 
    \end{align*}
    is a reduced word for $w(\lambda, a, b, n)$ and is braid-free.

\subsection{Tableaux}
A \defn{strict partition} is a descending sequence $\lambda = (\lambda_1 > \lambda_2 > \dots > \lambda_k)$ of positive integers, let $\ell(\lambda)=k$ be the number of parts of $\lambda$.
Corresponding to each partition, $\lambda$, is the \defn{Shifted Young Diagram} of $\lambda$ defined as the set:
\begin{equation*}
    D_{\lambda} = \{(i,j) \in \mathbb{Z}_+^2: i \leq j \leq \lambda_i + j\}.
\end{equation*}
A strict partition $\mu$ is contained in $\lambda$, denoted $\mu \subseteq \lambda$, if for all $i$ $\mu_i \leq \lambda_i$.
For such cases $\lambda / \mu$ is a skew shape with associated skew shifted Young diagram $D_{\lambda / \mu} = D_{\lambda} / D_{\mu}$.
Let $\theta$ be a shifted skew diagram.
    $\theta$ is a \defn{row} if all its boxes are in the same row.
    Likewise, $\theta$ is a \defn{column} if all its boxes are in the same column.
    $\theta$ is a \defn{rim} if it contains no set of four boxes that form a connected square.
    If $\theta$ is a rim, then its \defn{north-east arm} is the largest row or column that can be formed by intersecting $\theta$ with a square whose upper right corner is the most north-east box in $\theta$.
    Let $\hat{\theta}$ be the diagram obtain by removing the \defn{north-east arm} from $\theta$.

\begin{figure}[h!]
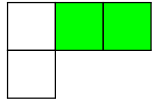

    \centering
    \begin{ytableau}
        \none & & *(green) & *(green)\\
        \none &
    \end{ytableau}
    \caption{A rim of skew shape $(4,1)/(3)$ whose north-east arm is the two rightmost boxes in the first row.}
    \label{fig:placeholder}
\end{figure}

A special partition shape is a \defn{shifted trapezoid} which for positive integers $a\leq b$ is defined as
\begin{equation}
    \tau(a,b) = (b+a-1, b+a-3,\ldots, b-a+1).
\end{equation}
We say $\tau(a,b)$ has height $a$ and top width $b+a-1$.
Note that when $a=1$, $\tau(a,b)$ is a single row.\\

\begin{figure}[h!]
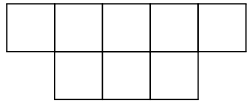

    \centering
    \ydiagram{0+5,1+3}
    \caption{A tableau of shape $\tau(2,4)$}
    \label{fig:placeholder}
\end{figure}

    Let $\lambda$ be a strict partition with $\ell(\lambda) = k$.
    A \defn{strict decomposition tableau} (SDT) is 
    a tableau $T:D_\lambda \rightarrow \mathbb{N}$ with rows $R_1 \dots R_k$ so that:
    \begin{enumerate}[(a)]
        \item For all $i \in [k]$, the row $R_i$ is unimodal.
        \item For $i \in [k-1]$, the first and last entries of $R_{i+1}$ are less than the first entry of $R_i$.
        \item \label{sdt.c} For $i \in [k-1]$, consider the increasing sequences $\mathsf{T}(R_i) = a_1 \dots a_{\lambda_i}$ and $\mathsf{B}(R_{i+1}) = b_1 \dots b_{\lambda_{i+1}}$.
        For all $j \in [\lambda_{i+1}]$,
        \begin{equation}
        \label{eq:row-rule}
        \{ \pm a_1, \dots, \pm a_j,\pm b_{j+1},\dots,\pm b_{\lambda_{i+1}}\} \cap (b_j,a_{j+1}] = \varnothing.    
        \end{equation}
    \end{enumerate}
    Where $\mathsf{T}(R)$ and $\mathsf{B}(R)$ are defined as the following:
    \begin{align*}
\mathsf{T}(R) &:= -r_1 < \dots < -r_{j-1} < r_j < r_{j+1} < \dots < r_k,\ \text{or}\\
\mathsf{B}(R) &:= -r_1 < \dots < -r_{j-1} < -r_j < r_{j+1} \dots < r_k,
\end{align*}
    In Condition (c), note that $a_{j+1}$ appears immediately above $b_j$ in $T$. If $a_{j+1} \leq b_j$, then we will have that $(b_j,a_{j+1}] = \varnothing$, so Equation~\eqref{eq:row-rule} is satisfied vacuously.
If $x$ is an element of the set defined in the LHS of Equation~\eqref{eq:row-rule}, we say $x$ \defn{witnesses} $b_j < a_{j+1}$, hence the failure of Condition (\ref{sdt.c}).

\begin{figure}[h!]
\begin{ytableau}
        4 & 1 & 0 & 3\\
        \none & 2 & 3\\
\end{ytableau}
\begin{ytableau}
        4 & 0 & 1 & 3\\
        \none & 2 & 3\\
\end{ytableau}
    \begin{ytableau}
        4 & \textcolor{darkred}{1} & 0 & 3\\
        \none & \textcolor{darkred}{2} & \textcolor{darkred}{1}\\
\end{ytableau}
    \begin{ytableau}
        4 & \textcolor{darkred}{3} & 1 & \textcolor{darkred}{3}\\
        \none & 1 & 0 & \textcolor{darkred}{2}\\
\end{ytableau}
\caption{The two left tableau are examples of strict decomposition tableau while the right two tableau are non-examples (the offending entries colored in red).}
\end{figure}

The \defn{reading word} of an SDT, $T$, is the $0$-Hecke expression created by reading of entries of $T$ bottom to top, left to right, denoted $\rho(T)$.

\begin{figure}[h!]
\begin{ytableau}
        1 & 0 & 2
\end{ytableau}\;\;\;
\begin{ytableau}
        2 & 0\\
        \none & 1
\end{ytableau}\;\;\;
\begin{ytableau}
        2 & 0 & 2\\
        \none & 1
\end{ytableau}
\caption{These are all the SDT with reading words equivalent to $\overline{2}31$.}
\end{figure}

\begin{lemma}\label{lem:config}
    A shifted tableau with unimodal rows is a strict decomposition tableau if and only if the tableau avoids the following five configurations:
\[
(i)\;\begin{ytableau}
    a & \cdots & \\
    \none & \cdots & b
\end{ytableau},
(ii)\;\begin{ytableau}
    \; & \cdots & a & \cdots & \\
    \none & \cdots & c & \cdots & b
\end{ytableau},
(iii)\;\begin{ytableau}
    \; & \cdots & v & z & \cdots & \\
    \none & \cdots & \cdots & x & \cdots & y
\end{ytableau},\\
\]\[
(iv)\;\begin{ytableau}
    y & \cdots & z\\
    \none & \cdots & x
\end{ytableau},
(v)\;\begin{ytableau}
    \; & \cdots & y & \cdots & z\\
    \none & \cdots & \; & \cdots & x
\end{ytableau}
\]
with $a\leq b < c$, $x < y \leq z$, and $v < z$.
\end{lemma}

\noindent For the proof of Lemma~\ref{lem:config} see \cite{arroyo2025typeckstanleysymmetric}.

\subsection{GQ Functions}

For $\lambda$ a strict partition, $GQ_\lambda$ is a symmetric function defined in terms of shifted set-valued tableaux and represents $K$--theory classes for the Lagrangian Grassmannian~\cite{ikeda2013k}.
In \cite{arroyo2025typeckstanleysymmetric}, we proved that for $\lambda$ a strict partition with $\lambda_1 < n$ and $a \leq b$ positive integers,
\[
G^C_{w(\lambda,a,b,n)} = GQ_\lambda \cdot GQ_{\tau(a,b)}.
\]
Where $G^C_w$ is the Type $C$ $K$-Stanley for $w$ which were introduced in \cite{kirillov2017construction}.
Additionally, in \cite{arroyo2025typeckstanleysymmetric} SDT were the insertion objects of an insertion method that gives the following conjecture:
\begin{conjecture}
	\label{conj:expansion}
	For $w$ a signed permutation,
	\begin{equation}
	G^C_w = \sum_{\lambda \ \mathrm{strict}} \beta^{|\lambda| - \ell(w)}a^C_w(\lambda) \cdot GQ_\lambda.
\label{eq:GC-conj}		
	\end{equation}
\end{conjecture}

This conjecture combined with the earlier result gives the following conjecture for expanding products of $GQ$'s:

\trapezoidConjecture*

When restricting to $a=1$, $\tau(1,b)$ is a rectangle meaning that conjecture \ref{conj:trapezoid} becomes a Pieri rule for $GQ$'s.
In \cite{buch2012pieri}, Buch and Ravikumar prove the following Pieri rule for $GQ$'s.

\begin{theorem}\label{thm:buchravikumar}
    \cite{buch2012pieri}
    For a strict partition $\lambda$ and $p\in\mathbb{Z}$,
    \begin{equation}
        GQ_{\lambda} \cdot GQ_p = \sum_{\mu} c^{\mu}_{\lambda,p} \beta^{|\mu|-|\lambda|-p} GQ_{\mu},
    \end{equation}
    with $c^{\mu}_{\lambda,p} = C(\theta,p)$ where $\theta = \mu \setminus \lambda$ be a shifted skew diagram and $C(\theta, p)$ is defined by the following:
    \begin{enumerate}[(a)]
        \item If $\theta$ is not a rim, then $C(\theta, p) = 0$. If $p \leq 0$, then $C(\theta, p) = \delta_{\theta, \emptyset}$. If $p > 0$, then $C(\emptyset, p) = 0$. For the rest of the rules assume $\theta \neq \emptyset$ is a rim and $p > 0$.
        \item If $\hat{\theta} = \emptyset$ and $\theta$ meets the diagonal, then
        \begin{equation*}\label{eqn:diagonal}
            C(\theta, p) = \begin{cases}
                \delta_{p, |\theta|} + \delta_{p, |\theta|-1} & \textrm{if $\theta$ is a column,}\\
                \delta_{p, |\theta|} & \textrm{if $\theta$ is a row.}
            \end{cases}
        \end{equation*}
        \item If $\hat{\theta} = \emptyset$ and $\theta$ is disjoint from the diagonal, then
        \begin{equation*}\label{eqn:notdiagonal}
            C(\theta, p) = 2\delta_{p,|\theta|} + \delta_{p, |\theta|-1}.
        \end{equation*}
        \item If $\hat{\theta}\neq \emptyset$ and $\theta \setminus \hat{\theta}$ is connected to $\hat{\theta}$ with $|\theta \setminus \hat{\theta}| = a$, then
        \begin{equation*}\label{eqn:connected}
            C(\theta, p) = C(\hat{\theta}, p-a) + C(\hat{\theta}, p-a+1).
        \end{equation*}
        \item If $\hat{\theta}\neq \emptyset$ and $\theta \setminus \hat{\theta}$ is not connected to $\hat{\theta}$ with $|\theta \setminus \hat{\theta}| = 1$, then
        \begin{equation*}\label{eqn:onebox}
            C(\theta, p) = 2 C(\hat{\theta}, p-1) + 2 C(\hat{\theta}, p).
        \end{equation*}
        \item If $\hat{\theta}\neq \emptyset$ and $\theta \setminus \hat{\theta}$ is not connected to $\hat{\theta}$ with $|\theta \setminus \hat{\theta}| = a > 1$, then
        \begin{equation*}\label{eqn:disconnected}
            C(\theta, p) = 2 C(\hat{\theta}, p-a) + 3 C(\hat{\theta}, p-a+1) + C(\hat{\theta}, p-a+2).
        \end{equation*}
    \end{enumerate}
\end{theorem}

The next section proves that when Conjecture \ref{conj:trapezoid} is restricted to $a=1$ it agrees with Theorem \ref{thm:buchravikumar}.

\section{Proof}

\begin{definition}
    Let $\textrm{SDT}(\mu, \lambda, p)$ be the set of strict decomposition tableaux of shape $\mu$ with reading word equivalent to $w(\lambda,1,p)$.
\end{definition}

\begin{definition}
    For tableau in the set $\textrm{SDT}(\mu, \lambda, p)$ we distinguish two different kinds of numbers that fill the tableau.
    A \defn{large number} is any number greater than $\lambda_1$.
    A \defn{small number} is any $\lambda_i - 1$ not in row $i$.
    A small number can be \defn{converted} into a large number to give a different tableau by becoming the smallest large number and then increasing all the other large numbers.
    Note that when examples of tableaux are given later in this section large numbers will be colored green while small numbers will be colored blue.
\end{definition}

We will now proceed to show that $|\textrm{SDT}(\mu,\lambda,p)|$ follows the same recursion as in Theorem~\ref{thm:buchravikumar}.
As an outline to the argument we will show that each row above the bottom row of the top most column is fixed and everything below that row must be an SDT with skew shape missing the northwest arm.
Then the amount of ways to create the distinguished row exactly follows the recursion in Theorem~\ref{thm:buchravikumar}.

\begin{lemma}\label{tabsetup}
For any $\textrm{SDT}(\mu, \lambda, p)$, the elements $\lambda_k-2,\ldots, 0$ must be in row $k$ consecutively.
\end{lemma}

\begin{proof}
    Note each $\lambda_k$ contributes $\lambda_k-1, \ldots, 0$ to the reading word.
    Suppose the $0$ from $\lambda_k$'s part appears above row $k$, then $0$ from $\lambda_{k-1}$ must be above row $k-1$ as two $0$s cannot be in the same row.
    We can repeat this until we require that that the $0$ from $\lambda_1$'s part is above row $1$, which is impossible.\\
    Suppose the $\lambda_k-2$ from $\lambda_k$ appears in row $\ell$ where $\ell > k$ and suppose $k$ is the smallest value for which this is true.
    Note that $\lambda_k-2$ cannot commute past $\lambda_k-1$, thus $\lambda_k-1$ must also appear in row $\ell$ or a lower row, further if $\lambda_k-1$ is in row $\ell$ it must be left of $\lambda_k-2$.
    Therefore, row $\ell$ must start with an element greater than or equal to $\lambda_k-1$ to avoid configuration $(i)$ if $\lambda_k-1$ is below row $\ell$ or as $\lambda_k-1$ is part of the decreasing portion of the row.
    Further, the element above the start of row $\ell$ must be in the greater than or equal to $\lambda_k-1$ in order to avoid configuration $(ii)$.
    Additionally, the start of row $\ell-1$ must be greater than the element to its right as otherwise the tableau forms configuration $(iv)$.
    As no copy of $\lambda_k-1$ can commute past the $\lambda_k-2$ in row $\ell$ and by assumption $k$ is smallest such that $\lambda_k-2$ appears below its row, we have that the element above the start of row $\ell$ is either a large number or $\lambda_i-1$ for some $i > k$.
    However, in the former case the front of row $\ell-1$ would have to be a bigger large number which cannot happen as large numbers cannot commute.
    In the latter case, the existence of $\lambda_i-1$ in row $\ell-1 \leq k < i$ implies there is at least one row above row $\ell-1$ in the tableau.
    Further, as $\lambda_j-1$s cannot commute past each other (due to the other terms $\lambda_j$ contributes), we have that the front row $\ell-1$ must be a large number.
    However, this then requires that the front of row $\ell-2$ is two large numbers in decreasing order otherwise the tableau would form configurations $(i)$ or $(iv)$, by similar arguments for row $\ell-1$.
    This is a contradiction as large numbers must be in increasing in a row as they cannot commute past each other.
    Therefore, the $\lambda_k-2$ that $\lambda_k$ contributes must be in row $k$.
    As the $\lambda_k-2$ from $\lambda_k$ cannot be moved below row $k$ and the $0$ from $\lambda_k$ cannot be moved above row $k$ and $\lambda_k-3, \ldots, 1$ cannot commute past $\lambda_k-2$ and $0$ all of $\lambda_k-2,\ldots, 0$ must appear in row $k$.
    Further, as the row is unimodal they must all be adjacent.
\end{proof}

\begin{lemma}\label{lem:large_inc}
    Large/small numbers must appear in weakly increasing order from bottom to top, left to right.
\end{lemma}

\begin{proof}
    Large numbers must appear in weakly increasing order as large numbers cannot commute past in other in the $0$-Hecke monoid.
    
    Suppose a large number appears before a small number on the same row.
    Let $i$ be the smallest index of such a row, note that $i>1$ as small numbers cannot appear in the first row as that would require them to commute past copies of themselves minus $1$ due to Lemma~\ref{tabsetup}.
    Note that to avoid configuration ($i$), the first element of row $i-1$ must be a large number.
    Additionally, the large number in row $i$ must be the first element of the row as the large number appears before the small number and the row is unimodal.
    The second element in row $i-1$ cannot be a small number otherwise, $i$ wouldn't be the smallest row index with a large number followed by a small number.
    Further, the second element in row $i-1$ cannot be a large number by Lemma~\ref{tabsetup}.
    Therefore, the tableau has an instance of configuration ($ii$) with $a$ being the second element in row $i-1$, $b$ being the small number in row $i$, and $c$ being the large number in row $i$.

    Suppose a large number appears below a small number.
    Let the large number be in row $i$ and the small number in row $j$.
    Then in order to have the tableau avoid configuration ($i$) we must have the first element in each row up to row $j$ be a large number.
    However, we are then in the previous case.
    
    Suppose a small number $s$ appears before another small number $r$ with $s>r$.
    First suppose both appear in the same row.
    Note that small numbers cannot appear in the first row as that would require them to commute past copies of themselves minus $1$ due to Lemma~\ref{tabsetup}.
    Let $i$ be the smallest index with such an $s$ and $r$ and take the leftmost instance of such a pair.
    Note that $s$ must be the first element of row $i$ in order for the row to be unimodal and to not be in one of the previous cases.
    The second element of row $i-1$ cannot be a large number as that would require the first entry to be a bigger large number, which then requires the entire row to be a large number contradicting Lemma~\ref{tabsetup}.
    Further, the second element of row $i-1$ cannot be a small number as then the first element of row $i-1$ would have to be a bigger small number, contradicting that $i$ is the smallest index, or be a large number, which is dealt with in a previous case.
    Therefore, the element above $s$ must be smaller than $s$ and $r$, meaning the tableau forms configuration ($ii$) with $a$ being the second element in row $i-1$, $b=r$, and $c=s$.
    Next, suppose $r$ appears above $s$, with $r$ in row $i$.
    Then in order to avoid configuration ($i$) we must have the first entry of row $i$ be larger than $s$.
    However, such an entry must either be a larger number or a bigger small number than $r$ meaning we are in a previous case.
\end{proof}

\begin{collary}
    Any $T\in\textrm{SDT}(\mu, \lambda, p)$ is completely determined by the location of large/small numbers and which ones are duplicates.
\end{collary}

\begin{lemma}\label{lem:toprows}
    If $\mu\neq \lambda$ and $\mu_i = \lambda_i$, then row $i$ of any $T\in\textrm{SDT}(\mu, \lambda, p)$ has exactly one large/small number which appears at the start of the row.
\end{lemma}
\begin{proof}
    By Lemma~\ref{tabsetup} the row needs to contain $\lambda_i-2,\ldots, 0$ that leaves only one free entry.
    This entry must be a large/small number as below row $i$ there is a row $j$ with $\mu_j > \lambda_j$ which requires the row to contain a large/small number which then requires row $i$ to also contain a large/small number in the front in order to avoid configuration ($i$).
\end{proof}

\begin{lemma}\label{notrim}
    If $\mu / \lambda$ is not a rim, then $\textrm{SDT}(\mu, \lambda, p)$ is $\emptyset$ for all $p\in \mathbb{Z}^+$.
\end{lemma}
\begin{proof}
    Suppose $\mu / \lambda$ is not a rim.
    Consider any $T\in\textrm{SDT}(\mu, \lambda, p)$.
    Then by definition there is a row $r$ such that $\mu_r - \lambda_r \geq 2$ and $\mu_{r+1} - \lambda_{r+1} \geq 2$.
    By Lemmas~\ref{tabsetup} and ~\ref{lem:large_inc} we have that there can only be one large/small number at the front of each row.
    Therefore, as $\mu_r - \lambda_r \geq 2$ and $\mu_{r+1} - \lambda_{r+1} \geq 2$ both rows must have a large/small number at the end.
    Further, as $\lambda_{r+1} < \lambda_r$ we must have that there is a large/small number in row $r$ directly above a large/small number in row $r+1$.
    However, as the large/small numbers are weakly increasing we have that the tableau forms configuration $(iv)$ with the first large/small number in row $r$ and the two large/small numbers on top of each other.
    Therefore, no such tableau exists.
\end{proof}

\begin{lemma}\label{lem:colDiagonal}
    $|\textrm{SDT}(\mu, \lambda, p)|=\delta_{p,|\theta|}+\delta_{p,|\theta|-1}$ if $\theta$ is a column touching the diagonal.
    Namely, the the only possible SDT is made by placing the large/small numbers on the diagonal up to the last row with an amount of small numbers equal to the number of parts where $\lambda_i=\mu_i$ if $|\theta|=p$ and one more if $|\theta| = p+1$.
\end{lemma}

\begin{proof}
    Suppose there was a row with two large/small numbers, $k$ and $\ell$ with $k < \ell$.
    However, then the tableau would from configuration $(iv)$ with $y=k$, $z=\ell$ and $x$ the element below $\ell$ due to Lemma~\ref{lem:large_inc} and needing to avoid configuration ($i$).
    Therefore, we have that all the large/small numbers must appear at the start of each row.
    Further, this means each row $i$ contains: $\mu_i-2, \mu_i-3, \ldots, 0$.
    This then means that we can only have small numbers corresponding to rows where $\mu_i = \lambda_i$ or $\lambda_i = 1$.
    Note that as $\theta$ is column touching the diagonal, the number of parts of $\mu$ is greater than the number of parts of $\lambda$.
    This means that the last row must start with a large/small number which then means each row must start with a large/small number in order to avoid configuration ($i$).
    Therefore, we must have small numbers corresponding to each row where $\mu_i = \lambda_i$ and we may have the last row end in the small number $0$ corresponding to the last part of $\lambda$.
    As the number of rows of the tableau is equal to $|\theta|$ plus the number of rows where $\mu_i = \lambda_i$ we can either have $|\theta|$ large numbers (where the last row isn't $0$) or $|\theta|-1$ large numbers (where the last row is $0$).
    Note that both of these tableau are in fact SDTs as the only possible way condition ($c$) is violated is if a $\mu_i-2$ above a small/large number witnesses something, however everything in between is either less than $\mu_i-2$ or a bigger small/large number.
\end{proof}

\begin{example}
    For $\lambda = (6,2,1)$ and $\mu = (6,3,2,1)$, $\theta$ is a column touching the diagonal of size $3$.
    Thus, $|\SDT{(\mu, \lambda, 3)}|=|\SDT{(\mu, \lambda, 2)}|=1$, with the tableaux witnessing this fact pictured below.

\begin{center}
\begin{ytableau}
*(green)9&4&3&2&1&0\\
\none & *(green) 8& 1 & 0\\
\none & \none & *(green) 7 & 0\\
\none & \none & \none & *(cyan) 5
\end{ytableau}\hspace{5em}
\begin{ytableau}
*(green)8&4&3&2&1&0\\
\none & *(green) 7& 1 & 0\\
\none & \none & *(cyan) 5 & 0\\
\none & \none & \none & *(cyan) 0
\end{ytableau}
\end{center}
    
\end{example}

\begin{lemma}\label{lem:rowDiag}
    $|\textrm{SDT}(\mu,\lambda,p)|=\delta_{p,|\theta|}$ if $\theta$ is a row touching the diagonal.
    Namely, the only possible SDT is made by placing the large numbers on the diagonal in decreasing order, then place the rest of the large numbers and small numbers in the last row in increasing order.
\end{lemma}

\begin{proof}
    By the same arguments as the last proof we require that each row begins with a large/small number, further by Lemma~\ref{tabsetup} for each row above $\theta$ there can only be one large/small number.
    Note that if $p > |\theta|$, there is not enough space in the tableau to get the right reading word.
    As $\lambda$ is the same as $\mu$ besides the last part, we have that each row besides the last row must contain $\lambda_r - 2$ through $0$, thus each row only has one free slot.
    Therefore, if $p < |\theta|$ we require a duplicate large or small number in the last row.
    However, if a large number is in the front of the last row it will witness the small numbers in the row.
    Further, our reading word requires that the large/small numbers appear in increasing order, thus if we have a repeat the row is not unimodal.
    Finally, when $p = |\theta|$ we have that the tableau is an SDT as the numbers are in increasing order and thus cannot be witnesses to each other.
    Further, they would be witnesses to each other if they didn't appear in increasing order, giving us that this is the only SDT with the right reading word.
\end{proof}

\begin{example}
    For $\lambda = (5,4)$ and $\mu = (5,4,3)$, $\theta$ is a row touching the diagonal of size $3$.
    Thus, $|\SDT{(\mu, \lambda, 3)}|=1$, with the tableau witnessing this fact pictured below.

\begin{center}
\begin{ytableau}
*(green)8&3&2&1&0\\
\none & *(green) 7& 2 & 1 & 0\\
\none & \none & *(cyan) 3 & *(cyan) 4 & *(green) 6
\end{ytableau}
\end{center}
    
\end{example}

\begin{lemma}\label{lem:colNotDiag}
    $|\textrm{SDT}(\mu,\lambda,p)|=2\delta_{p,|\theta|}+\delta_{p,|\theta|-1}$ if $\theta$ is a column not touching the diagonal.
    Namely, the only possible SDT are made by placing the large/small numbers on the diagonal in decreasing order up to the last row of $\theta$.
    For that last row you may place a large/small number at the front if $p=|\theta|$ or at the back.
\end{lemma}

\begin{proof}
    Suppose there was a row with two large/small numbers, $k$ and $\ell$ with $k < \ell$.
    However, then the tableau would from configuration $(iv)$ with $y=k$, $z=\ell$ and $x$ the element below $\ell$ due to Lemma~\ref{lem:large_inc} and needing to avoid configuration ($i$).
    Therefore, we have that all the large/small numbers must appear at the start of each row with the possible exception of the last row of $\lambda$ as there are no more large/small numbers below that row.
    Further, this means each row $i$ contains: $\mu_i-2, \mu_i-3, \ldots, 0$.
    This then means that we can only have front small numbers corresponding to rows where $\mu_i = \lambda_i$.
    If $|\theta|=p$, then the large/small number in the last row of $\theta$ either corresponds to a large number or a small number from a row where $\mu_i = \lambda_i$, either was this number is guaranteed to be at least $2$ larger than the biggest number in the row.
    Therefore, the large number can either appear at front or end of the row, corresponding to two different SDT.
    If $|\theta|=p+1$, then the small number in the last row of $\theta$ must be $\lambda_{\ell-1}-1$ where $\lambda_{\ell}$ is the last part of $\lambda$ as all the other large/small numbers have already been placed and configuration ($i$) prevents any duplicates.
    This forces the small number to be at the end of row as $\lambda_{\ell-1}-1$ cannot commute with $\lambda_{\ell}-1$ because due to $\theta$ being a column, $\lambda_{\ell-1}-1 = \lambda_{\ell}$.
    Therefore for $|\theta|=p+1$, there is only SDT.
    Note that these tableau are in fact SDTs as the only possible way condition ($c$) is violated is if a $\mu_i-2$ above a small/large number witnesses something, however everything in between is either less than $\mu_i-2$ or a bigger small/large number.
\end{proof}

\begin{example}
    For $\lambda = (6,4,3,1)$ and $\mu = (6,5,4,1)$, $\theta$ is a column not touching the diagonal of size $3$.
    Thus, $|\SDT{(\mu, \lambda, 3)}|=2$ and $|\SDT{(\mu, \lambda, 2)}|=1$, with the tableaux witnessing this fact pictured below.

\begin{center}
\begin{ytableau}
*(green)8&4&3&2&1&0\\
\none & *(green) 7& 3 & 2 & 1 & 0\\
\none & \none & *(cyan) 5 & 2 & 1 & 0\\
\none & \none & \none & 0
\end{ytableau}
\begin{ytableau}
*(green)8&4&3&2&1&0\\
\none & *(green) 7& 3 & 2 & 1 & 0\\
\none & \none & 2 & 1 & 0 & *(cyan) 5\\
\none & \none & \none & 0
\end{ytableau}\hspace{1em}
\begin{ytableau}
*(green)7&4&3&2&1&0\\
\none & *(cyan) 5& 3 & 2 & 1 & 0\\
\none & \none & 2 & 1 & 0 & *(cyan) 3\\
\none & \none & \none & 0
\end{ytableau}
\end{center}
    
\end{example}

\begin{lemma}\label{lem:rowNotDiag}
    $|\textrm{SDT}(\mu,\lambda,p)|=2\delta_{p,|\theta|}+\delta_{p,|\theta|-1}$ if $\theta$ is a row not touching the diagonal.
    Namely, the only possible SDT is made by placing the large/small numbers on the diagonal in decreasing order up to the row of $\theta$.
    For that row you may place $|\theta|$ large/small numbers with none at the front if $p=|\theta|$ or one in the front (being a duplicate if $p=|\theta|-1$).
\end{lemma}

\begin{proof}
    For each row above $\theta$ must begin with a large/small number and can only have one large/small number due to Lemma~\ref{tabsetup}.
    Note that having a large/small number at the start of each of those rows forces there to be a small number corresponding to that row somewhere below.
    Any row below $\theta$ cannot contain a large/small number as in order to have the correct reading word each $\lambda_i-1$, for row $i$ below $\theta$, must also be below $\theta$.
    This then forces each row $i$ below $\theta$ to be $\lambda_i-1, \lambda_i-2, \ldots, 0$.
    As rows below $\theta$ cannot contain small numbers this means that $\lambda_j-1$ must be in row $j$ where $j$ is the row of $\theta$, therefore row $j$ has $|\theta|$ large/small numbers.
    If $|\theta| = p$, then the smallest large/small number can either be at the front or the back of the row, giving two SDTs.
    If $|\theta| = p-1$, then a large/small number must be duplicated which has to be the smallest large/small number in row $j$ and must be at the front and back of the row.
    Note that these tableau are in fact SDTs as the only possible way condition ($c$) is violated is if a $\mu_i-2$ above a small/large number witnesses something, however everything in between is either less than $\mu_i-2$ or a bigger small/large number.
\end{proof}

\begin{example}
    For $\lambda = (6,4,3,1)$ and $\mu = (6,5,4,1)$, $\theta$ is a row not touching the diagonal of size $3$.
    Thus, $|\SDT{(\mu, \lambda, 3)}|=2$ and $|\SDT{(\mu, \lambda, 2)}|=1$, with the tableaux witnessing this fact pictured below.

\begin{center}
\begin{ytableau}
*(green)9&4&3&2&1&0\\
\none & 1& 0 & *(cyan) 5& *(green) 7 & *(green) 8\\
\none & \none & 0
\end{ytableau}
\begin{ytableau}
*(green)9&4&3&2&1&0\\
\none & *(cyan) 5& 1 & 0& *(green) 7 & *(green) 8\\
\none & \none & 0
\end{ytableau}\hspace{1em}
\begin{ytableau}
*(green)8&4&3&2&1&0\\
\none & *(cyan)5 & 1& 0 & *(cyan) 5& *(green) 7\\
\none & \none & 0
\end{ytableau}
\end{center}
    
\end{example}

\begin{lemma}\label{lem:rowRecursion}
    If $\theta/\hat{\theta}$ is a row then $\textrm{SDT}(\mu, \lambda, p)$ follows the recursion in Theorem \ref{thm:buchravikumar}, where $\theta = \mu \setminus \lambda$. 
\end{lemma}

\begin{proof}
    Consider a tableau in $\textrm{SDT}(\hat{\mu}, \lambda, p)$ we can transform this tableau into an element of $\textrm{SDT}(\mu, \lambda)$, where $\mu -\lambda$ and $\hat{\mu}-\lambda$ is a rim, and $\mu-\hat{\mu}$ is a row of $a$ boxes in row $r$, in one of the four following ways:
\begin{enumerate}
    \item Add in $a$ new large numbers to the end of row $r$. Note that when adding in new large numbers if there are small numbers above row $r$ swap the added large numbers with the small numbers above.
    \item Add in $\lambda_r-1$ to the second space of row $r$, shifting the other numbers right, and put $a-1$ new large numbers to the end of row $r$.
    \item Add a duplicate of the front entry and $a-1$ new large numbers to the end of row $r$.
    \item Add in $\lambda_r-1$ to the second space of row $r$, shifting the other numbers right, and a duplicate of the front entry and $a-2$ new large numbers to the end of row $r$.
\end{enumerate}
Further, for each case where we add in $\lambda_r-1$ into row $r$ we may change copies of $\lambda_r-1$ below row $r$ into large numbers to increase our large number count by one.
Note that we can only add in $\lambda_r-1$ to the row if $\mu-\hat{\mu}$ is disconnected from $\hat{\mu}-\lambda$ as otherwise $\lambda_r-1$ is already in row $r$.
Additionally, note that we can only add in $\lambda_r-1$ and a duplicate of the front entry if $a > 1$ as it requires adding at least two boxes.

To see that all of these created tableau are in fact SDT, first consider the numbers added to the end of row $r$.
By definition of north-east arm we have that there are no boxes below the added boxes, therefore if they create a configuration from Lemma~\ref{lem:config} the new numbers must serve the role as $b$, $y$, or $x$.
The large number at the front of row $r-1$ prevents configuration ($i$) from forming.
The lact of any other large numbers in row $r-1$ means configuration ($iii$), ($iv$), and ($v$).
Finally configuration ($ii$) cannot occur as all numbers left of any of the added large numbers must be smaller.
Next consider added $\lambda_r-1$ and the other shifted numbers in row $r$.
Note that if $r > 1$, then above $\lambda_r-i$ is $\lambda_{r-1}-i$ and that row $r-1$ is decreasing at least up to $\lambda_{r-1}-i$, thus $\lambda_{r}-i$ cannot form a configuration as part of the bottom row of the configuration.
Further, as everything left of $\lambda_{r}-i$ is larger than it, $\lambda_{r}-i$ cannot form configurations ($iii$), ($iv$), or ($v$).
Additionally, if $\lambda_{r}-i$ is less than the element below it that requires that element to be a large number.
Thus, in order for $\lambda_{r}-i$ to form configuration ($ii$) we require $b$ and $c$ to be large numbers ($b$ is required as only $\lambda_{r}-1$ can be above a large number in the decreasing part of the row), however large numbers are increasing left to right.
Finally, consider large numbers below row $r$ that used to be copies of $\lambda_r-1$.
Note that these must appear at either the front on end of their row and the entry grew into a larger number.
Therefore, we only need to consider configurations ($i$), ($ii$), and ($iii$).
Configuration ($i$) cannot occur as the first entry of the next row is a large number (it cannot be a small number as $\lambda_r-1$ was the largest small number).
Configuration ($ii$) cannot occur as large numbers increase left to right.
Configuration ($iii$) cannot occur as in order to have $x < y < z$ where $y$ is our newly created large number we require that $x$ is not a large number, however this then requires that $x < \lambda_r-1$ as the largest small number was $\lambda_r-1$.
This would then mean that the configuration already occurred in our starting tableaux, a contradiction.

Thus, when $\hat{\theta}$ and $\theta / \hat{\theta}$ are connected only cases (1) and (3) apply meaning that:
\begin{equation}\label{eqn:row1}
    |\textrm{SDT}(\mu,\lambda,p)| \geq  |\textrm{SDT}(\hat{\mu},\lambda,p-a)| + |\textrm{SDT}(\hat{\mu},\lambda,p-a+1)|.
\end{equation}
When $\hat{\theta}$ and $\theta / \hat{\theta}$ are disconnected and $|\theta / \hat{\theta}|=1$ cases (1), (2), and (3) apply with (2) having two options, thus:
\begin{equation}\label{eqn:row2}
    |\textrm{SDT}(\mu,\lambda,p)| \geq  2|\textrm{SDT}(\hat{\mu},\lambda,p-1)| + 2|\textrm{SDT}(\hat{\mu},\lambda,p)|.
\end{equation}
Finally, when $\hat{\theta}$ and $\theta / \hat{\theta}$ are disconnected and $|\theta / \hat{\theta}|=a > 1$ all four cases apply with (2) and (4) having two options, thus:
\begin{equation}\label{eqn:row3}
    |\textrm{SDT}(\mu,\lambda,p)| \geq  2|\textrm{SDT}(\hat{\mu},\lambda,p-a)| + 3|\textrm{SDT}(\hat{\mu},\lambda,p-a+1)|+ |\textrm{SDT}(\hat{\mu},\lambda,p-a+2)|.
\end{equation}

For the inverse remove all the large numbers from row $r$ besides the left most one and remove $\lambda_r-1$ from row $r$, if it exists, shifting elements to it's right leftwards.
Then, if there is no $\lambda_r - 1$ below row $r$ convert the smallest large numbers to $\lambda_r -1$.
Finally, lower the existing large numbers as necessary.
As each row above row $r$ is forced to only have a large number on the front this map is in fact invertible, given the addition of a token saying which map to get from $\hat{\mu}$ to $\mu$.
As row $r$ is forced to only have $|\hat{\theta}|$ or $|\hat{\theta}|-1$ large/small numbers (due to Lemma~\ref{tabsetup}), this inverse map accounts for every possible SDT.
Further, note that if the resulting tableau is not an SDT the configuration witnessing it not being an SDT would have still existed in the starting tableau.
Therefore, equations (\ref{eqn:row1}), (\ref{eqn:row2}), and (\ref{eqn:row3}) all become equalities.
\end{proof}

\begin{example}
    For $\lambda = (6,2,1)$ and $\mu = (6,5,2)$, we have that $\theta-\hat{\theta}$ is a row of size $3$ that is connected to $\hat{\theta}$.
    Therefore, from elements of $\SDT{(\hat{\mu},\lambda,1)}$ for $\hat{\mu} = (6,2,1)$ we can create element of $\SDT{(\mu,\lambda,4)}$ from case $(1)$ and an element of $\SDT{(\mu,\lambda,3)}$ from case $(3)$, pictured below.
    \begin{center}
\ytableausetup{smalltableaux}
\begin{ytableau}
*(green) 7&4&3&2&1&0\\
\none & *(cyan) 5 & 1 & 0\\
\none & \none & 0 & *(cyan) 1
\end{ytableau}
$\rightarrow$
\begin{ytableau}
*(green) 9&4&3&2&1&0\\
\none & *(cyan) 5 & 1 & 0 & *(green) 7 & *(green) 8\\
\none & \none & 0 & *(cyan) 1\\
\none\\
\none & \none & \none & \none[(1)]
\end{ytableau}
$\mid$
\begin{ytableau}
*(green) 8&4&3&2&1&0\\
\none & *(cyan) 5 & 1 & 0 & *(cyan) 5 & *(green) 7\\
\none & \none & 0 & *(cyan) 1\\
\none\\
\none & \none & \none & \none[(3)]
\end{ytableau}
\end{center}
\ytableausetup{nosmalltableaux}
\end{example}

\begin{example}\label{ex:rowsize1}
    For $\lambda = (5,1)$ and $\mu = (6,3,1)$, we have that $\theta-\hat{\theta}$ is a row of size $1$ that is not connected to $\hat{\theta}$.
    Therefore, from elements of $\SDT{(\hat{\mu},\lambda,2)}$ for $\hat{\mu} = (5,3,1)$ we can create two elements of $\SDT{(\mu,\lambda,4)}$ from cases $(1)$ and $(2)^*$ and two elements of $\SDT{(\mu,\lambda,3)}$ from cases $(2)$ and $(3)$, pictured below.
    \begin{center}
\ytableausetup{smalltableaux}
\begin{ytableau}
*(green) 7&3&2&1&0\\
\none & *(green) 6 & 0 & *(Green) 6\\
\none & \none & *(cyan) 4
\end{ytableau}
$\rightarrow$
\begin{ytableau}
*(green) 7&3&2&1&0&*(green) 8\\
\none & *(green) 6 & 0 & *(Green) 6\\
\none & \none & *(cyan) 4\\
\none\\
\none & \none & \none[(1)]
\end{ytableau}
,
\begin{ytableau}
*(green)8&*(red)4&3&2&1&0\\
\none & *(green) 7 & 0 & *(Green) 7\\
\none & \none & *(green) 6\\
\none\\
\none & \none & \none[(2)^*]
\end{ytableau}
$\mid$
\begin{ytableau}
*(green)7&*(red)4&3&2&1&0\\
\none & *(green) 6 & 0 & *(Green) 6\\
\none & \none & *(cyan) 4\\
\none\\
\none & \none & \none[(2)]
\end{ytableau}
,
\begin{ytableau}
*(green) 7&3&2&1&0&*(Green) 7\\
\none & *(green) 6 & 0 & *(Green) 6\\
\none & \none & *(cyan) 4\\
\none\\
\none & \none & \none[(3)]
\end{ytableau}
\end{center}
\ytableausetup{nosmalltableaux}
\end{example}

\begin{example}
    For $\lambda = (4,1)$ and $\mu = (7,3,1)$, we have that $\theta-\hat{\theta}$ is a row of size $3$ that is not connected to $\hat{\theta}$.
    Therefore, from elements of $\SDT{(\hat{\mu},\lambda,1)}$ for $\hat{\mu} = (4,3,1)$ we can create two elements of $\SDT{(\mu,\lambda,4)}$ from cases $(1)$ and $(2)^*$, three elements of $\SDT{(\mu,\lambda,3)}$ from cases $(2)$, $(3)$, and $(4)^*$, and an element of $\SDT{(\mu,\lambda,3)}$ from case $(4)$, pictured below.
    \begin{center}
    \begin{tabular}{cccc}
\ytableausetup{smalltableaux}
\begin{ytableau}
*(green) 5&2&1&0\\
\none & *(cyan) 3 & 0 & *(cyan) 3\\
\none & \none & *(cyan) 0
\end{ytableau}
$\rightarrow$ &
\begin{ytableau} 
*(green) 5&2&1&0&*(green) 6 & *(green) 7 & *(green) 8\\
\none & *(cyan) 3 & 0 & *(cyan) 3\\
\none & \none & *(cyan) 0\\
\none\\
\none & \none & \none[(1)]
\end{ytableau}
,&
\begin{ytableau}
*(green) 6&*(red) 3 &2&1&0 & *(green) 7 & *(green) 8\\
\none & *(green) 5 & 0 & *(Green) 5\\
\none & \none & *(cyan) 0\\
\none\\
\none & \none & \none[(2)^*]
\end{ytableau}
\ $\mid$&
\begin{ytableau}
*(green) 5&*(red) 3 &2&1&0 & *(green) 6 & *(green) 7\\
\none & *(cyan) 3 & 0 & *(cyan) 3\\
\none & \none & *(cyan) 0\\
\none\\
\none & \none & \none[(2)]
\end{ytableau}
,\\\\
&\begin{ytableau}
*(green) 5&2&1&0&*(Green) 5 & *(green) 6 & *(green) 7\\
\none & *(cyan) 3 & 0 & *(cyan) 3\\
\none & \none & *(cyan) 0\\
\none\\
\none & \none & \none[(3)]
\end{ytableau}
,&
\begin{ytableau}
*(green) 6&*(red) 3 &2&1&0 & *(Green) 6 & *(green) 7\\
\none & *(green) 5 & 0 & *(Green) 5\\
\none & \none & *(cyan) 0\\
\none\\
\none & \none & \none[(4)^*]
\end{ytableau}
\ $\mid$&
\begin{ytableau}
*(green) 5&*(red) 3 &2&1&0 & *(Green) 5 & *(green) 6\\
\none & *(cyan) 3 & 0 & *(cyan) 3\\
\none & \none & *(cyan) 0\\
\none\\
\none & \none & \none[(4)]
\end{ytableau}
\end{tabular}
\end{center}
\ytableausetup{nosmalltableaux}
\end{example}

\begin{lemma}\label{lem:colRecursion}
    If $\theta/\hat{\theta}$ is a column then $\textrm{SDT}(\mu, \lambda, p)$ follows the recursion in Theorem \ref{thm:buchravikumar}, where $\theta = \mu \setminus \lambda$. 
\end{lemma}

\begin{proof}
    Consider a tableau in $\textrm{SDT}(\mu, \hat{\lambda}, p)$ we can transform this tableau into an element of $\textrm{SDT}(\mu, \lambda)$, where $\mu -\lambda$ and $\mu-\hat{\lambda}$ is a rim, and $\lambda-\hat{\lambda}$ is a column of $a$ boxes with lowest row $r$, in one of the three following ways:
\begin{enumerate}
    \item Convert small numbers to large numbers corresponding to the added column.
    \item Remove $\hat{\lambda}_r-1$ from row $r$, add in a new large number to the end of row $r$, and convert small numbers to large numbers corresponding to added column besides the bottom box.
    \item Remove $\hat{\lambda}_r-1$ from row $r$, add in a duplicate of the front of row $r$ to the end of row $r$, and convert small numbers to large numbers corresponding to added column besides the bottom box.
\end{enumerate}

Note that if $\hat{\lambda}-\lambda$ is connected to $\mu-\hat{\lambda}$, then we cannot change values of row $r$ besides the front number without breaking the SDT conditions, therefore only case (1) holds.
We can choose to unconvert the smallest converted small number to get a tableau with one less large number as long as the column has more than one box or we are in case (1).
Note that when there is no box below the end of row $r$ making the last box a large/small number to the end of a row cannot create a new configuration.
Similarly converting small numbers and shifting the row will not make new configurations.

Thus, when $\hat{\theta}$ and $\theta/\hat{\theta}$ are connected we have only case (1), but may choose to unconvert the smallest large number giving that:
\begin{equation}\label{e:col1}
    |\textrm{SDT}(\mu, \lambda, p)| \geq |\textrm{SDT}(\mu, \hat{\lambda},p-a)| + |\textrm{SDT}(\mu, \hat{\lambda},p-a-1)|.
\end{equation}

When $\hat{\theta}$ and $\theta/\hat{\theta}$ are not connected and $|\theta/\hat{\theta}|=1$ we have cases (1), (2), and (3) but may only choose to unconvert the smallest large number in case (1).
This gives that:
\begin{equation}\label{e:col2}
    |\textrm{SDT}(\mu, \lambda, p)| \geq 2|\textrm{SDT}(\mu, \hat{\lambda},p-a)| + 2|\textrm{SDT}(\mu, \hat{\lambda},p-a-1)|.
\end{equation}

Finally, When $\hat{\theta}$ and $\theta/\hat{\theta}$ are not connected and $|\theta/\hat{\theta}|>1$ we have all three cases and may choose to unconvert the smallest large number in each.
Therefore,
\begin{equation}\label{e:col3}
    |\textrm{SDT}(\mu, \lambda, p)| \geq 2|\textrm{SDT}(\mu, \hat{\lambda},p-a)| + 3|\textrm{SDT}(\mu, \hat{\lambda},p-a-1)|+ |\textrm{SDT}(\mu, \hat{\lambda},p-a-2)|.
\end{equation}

For the inverse, if there is a large number at the end of row $r$ remove it, shifted all other non-large numbers right one, and add in $\hat{\lambda}_r-1$ to the empty spot made.
Then convert all the large numbers corresponding to the added rows to small numbers.
Note this is in fact an inverse, given the addition of a token saying which map to get from $\hat{\lambda}$ to $\lambda$, as you can only have a large number at the end of row r in cases (2) and (3) which are distinguished by the existence of a duplicate.
Further, if the resulting tableau has a configuration it would necessarily mean the original tableau also has the configuration.
Therefore, equations (\ref{e:col1}), (\ref{e:col2}), and (\ref{e:col3}) become equalities.

\end{proof}

\begin{example}
    For $\lambda = (4,3,1)$ and $\mu = (5,4,3,1)$, we have that $\theta-\hat{\theta}$ is a column of size $2$ that is connected to $\hat{\theta}$.
    Therefore, from elements of $\SDT{(\mu,\hat{\lambda},2)}$ for $\hat{\lambda} = (5,4,1)$ we can create element of $\SDT{(\mu,\lambda,4)}$ from case $(1)$ and an element of $\SDT{(\mu,\lambda,3)}$ from case $(1)^*$, pictured below.
    \begin{center}
\ytableausetup{smalltableaux}
\begin{ytableau}
*(green) 7&3&2&1&0\\
\none & *(green) 6& 2 & 1 & 0\\
\none & \none & *(cyan) 4 & 0 & *(cyan) 4\\
\none & \none & \none & *(cyan) 3
\end{ytableau}
$\rightarrow$
\begin{ytableau}
*(green) 8&3&2&1&0\\
\none & *(green) 7& 2 & 1 & 0\\
\none & \none & *(green) 6 & 0 & *(Green) 6\\
\none & \none & \none & *(green) 5\\
\none \\
\none & \none & \none & \none[(1)]
\end{ytableau}
$\mid$
\begin{ytableau}
*(green) 7&3&2&1&0\\
\none & *(green) 6& 2 & 1 & 0\\
\none & \none & *(green) 5 & 0 & *(Green) 5\\
\none & \none & \none & *(cyan) 2\\
\none \\
\none & \none & \none & \none[(1)^*]
\end{ytableau}
\end{center}
\ytableausetup{nosmalltableaux}
\end{example}

\begin{example}
    For $\lambda = (5,1)$ and $\mu = (6,3,1)$, we have that $\theta-\hat{\theta}$ is a column of size $1$ that is not connected to $\hat{\theta}$.
    Therefore, from elements of $\SDT{(\mu,\hat{\lambda},2)}$ for $\hat{\lambda} = (6,1)$ we can create two elements of $\SDT{(\mu,\lambda,4)}$ from cases $(1)$ and $(2)^*$ and two elements of $\SDT{(\mu,\lambda,3)}$ from cases $(2)$ and $(3)$, pictured below.
    Note the tableaux created are the same ones from example~\ref{ex:rowsize1} as a single block can be consider a row or a column.
    \begin{center}
\ytableausetup{smalltableaux}
\begin{ytableau}
*(green) 8&4&3&2&1&0\\
\none & *(green) 7 & 0 & *(Green) 7\\
\none & \none & *(cyan) 5
\end{ytableau}
$\rightarrow$
\begin{ytableau}
*(green)8&4&3&2&1&0\\
\none & *(green) 7 & 0 & *(Green) 7\\
\none & \none & *(green) 6\\
\none\\
\none & \none & \none[(1)]
\end{ytableau}
,
\begin{ytableau}
*(green) 7&3&2&1&0&*(green) 8\\
\none & *(green) 6 & 0 & *(Green) 6\\
\none & \none & *(cyan) 4\\
\none\\
\none & \none & \none[(2)]
\end{ytableau}
$\mid$
\begin{ytableau}
*(green)7&4&3&2&1&0\\
\none & *(green) 6 & 0 & *(Green) 6\\
\none & \none & *(cyan) 4\\
\none\\
\none & \none & \none[(1)^*]
\end{ytableau}
,
\begin{ytableau}
*(green) 7&3&2&1&0&*(Green) 7\\
\none & *(green) 6 & 0 & *(Green) 6\\
\none & \none & *(cyan) 4\\
\none\\
\none & \none & \none[(3)]
\end{ytableau}
\end{center}
\ytableausetup{nosmalltableaux}
\end{example}

\begin{example}
    For $\lambda = (5,4,1)$ and $\mu = (6,5,3,1)$, we have that $\theta-\hat{\theta}$ is a row of size $2$ that is not connected to $\hat{\theta}$.
    Therefore, from elements of $\SDT{(\mu,\hat{\lambda},1)}$ for $\hat{\lambda} = (6,5,1)$ we can create two elements of $\SDT{(\mu,\lambda,3)}$ from cases $(1)$ and $(2)$, three elements of $\SDT{(\mu,\lambda,2)}$ from cases $(1)^*$, $(2)^*$, and $(3)$, and an element of $\SDT{(\mu,\lambda,1)}$ from case $(3)^*$, pictured below.
    \begin{center}
    \begin{tabular}{cccc}
\ytableausetup{smalltableaux}
\begin{ytableau}
*(green) 7&4&3&2&1&0\\
\none & *(cyan) 5&3&2&1 & 0\\
\none & \none & *(cyan) 4 & 0 & *(cyan) 4\\
\none & \none & \none & *(cyan) 0
\end{ytableau}
$\rightarrow$ &
\begin{ytableau}
*(green) 8&4&3&2&1&0\\
\none & *(green) 7&3&2&1 & 0\\
\none & \none & *(green) 6 & 0 & *(Green) 6\\
\none & \none & \none & *(cyan) 0\\
\none\\
\none & \none & \none & \none[(1)]
\end{ytableau}
,&
\begin{ytableau}
*(green) 8&4&3&2&1&0\\
\none & *(green) 6&2&1 & 0 & *(green) 7\\
\none & \none & *(cyan) 3 & 0 & *(cyan) 3\\
\none & \none & \none & *(cyan) 0\\
\none\\
\none & \none & \none & \none[(2)]
\end{ytableau}
\ $\mid$&
\begin{ytableau}
*(green) 7&4&3&2&1&0\\
\none & *(green) 6&3&2&1 & 0\\
\none & \none & *(cyan) 3 & 0 & *(cyan) 3\\
\none & \none & \none & *(cyan) 0\\
\none\\
\none & \none & \none & \none[(1)^*]
\end{ytableau}
,\\\\
&\begin{ytableau}
*(green) 7&4&3&2&1&0\\
\none & *(cyan) 4&2&1 & 0 & *(green) 6\\
\none & \none & *(cyan) 3 & 0 & *(cyan) 3\\
\none & \none & \none & *(cyan) 0\\
\none\\
\none & \none & \none & \none[(2)^*]
\end{ytableau}
,&
\begin{ytableau}
*(green) 7&4&3&2&1&0\\
\none & *(green) 6&2&1 & 0 & *(Green) 6\\
\none & \none & *(cyan) 3 & 0 & *(cyan) 3\\
\none & \none & \none & *(cyan) 0\\
\none\\
\none & \none & \none & \none[(3)]
\end{ytableau}
\ $\mid$&
\begin{ytableau}
*(green) 6&4&3&2&1&0\\
\none & *(cyan) 4&2&1 & 0 & *(cyan) 4\\
\none & \none & *(cyan) 3 & 0 & *(cyan) 3\\
\none & \none & \none & *(cyan) 0\\
\none\\
\none & \none & \none & \none[(3)^*]
\end{ytableau}
\end{tabular}
\end{center}
\ytableausetup{nosmalltableaux}
\end{example}

\arroyoPieri*

\begin{proof}
    Recall from Theorem~\ref{thm:buchravikumar} that $c^{\mu}_{\lambda,p}$ is the Littlewood-Richardson coefficient for multiplying $GQ_{\lambda}$ and $GQ_p$.
    If $\lambda/\mu$ is not a rim, then by Lemma~\ref{notrim} $|\SDT(\mu, \lambda, p)| = c^{\mu}_{\lambda,p} = 0$.
    Let $\lambda^0 \supseteq \lambda^1 \supseteq \cdots \supseteq \lambda^n$ and $\mu^0 \subseteq \mu^1 \subseteq \cdots \subseteq \mu^n$ be sequences of strict partitions such that:
    \begin{itemize}
        \item $\lambda^n = \lambda$ and $\mu^n = \mu$,
        \item $\mu^0 / \lambda^0$ is a single row or column,
        \item for each $i\in [n]$ either $\lambda^{i-1}=\lambda^{i}$ or $\mu^{i-1}=\mu^{i}$,
        \item for each $i\in [n]$, $\mu^i / \lambda^i - \mu^{i-1} / \lambda^{i-1}$ is the northeast arm of $\mu^i / \lambda^i$.
    \end{itemize}
    By Lemmas~\ref{lem:rowDiag},~\ref{lem:rowNotDiag},~\ref{lem:colDiagonal}, and ~\ref{lem:colNotDiag}, we have that $\SDT(\mu^0,\lambda^0,p)=c^{\mu}_{\lambda,p}$.
    Then by Lemmas~\ref{lem:rowRecursion} and ~\ref{lem:colRecursion} and induction we have that $\SDT(\mu^i,\lambda^i,p)=c^{\mu^i}_{\lambda^i,p}$ for each $i\in[n]$.
\end{proof}

\begin{remark}

It would be helpful to create some sort of formula for $|\SDT{(\mu, w(\lambda, a, b))}|$ to make the computation easier.
Some challenges with making such a formula is that large/small numbers are no longer in increasing order, numbers smaller than $\lambda_i-1$ can appear below row $i$, and $\mu / \lambda$ no longer needs to be a rim to get a nonzero amount of tableaux.

\begin{figure}[h!]
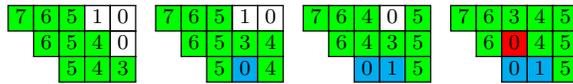

    \centering
    \ytableausetup{smalltableaux}
    \begin{ytableau}
*(green)7&*(green)6&*(green)5&1&0\\
\none & *(green) 6& *(green) 5 & *(green) 4& 0\\
\none & \none & *(green) 5 & *(green) 4 & *(green) 3
\end{ytableau}\;
    \begin{ytableau}
*(green)7&*(green)6&*(green)5&1&0\\
\none & *(green) 6& *(green) 5 & *(green) 3& *(green) 4\\
\none & \none & *(green) 5 & *(cyan) 0 & *(green) 4
\end{ytableau}\;
    \begin{ytableau}
*(green)7&*(green)6&*(green)4&0&*(green)5\\
\none & *(green) 6& *(green) 4 & *(green) 3& *(green) 5\\
\none & \none & *(cyan) 0 & *(cyan) 1 & *(green) 5
\end{ytableau}\;
    \begin{ytableau}
*(green)7&*(green)6&*(green)3&*(green)4&*(green)5\\
\none & *(green) 6& *(red) 0 & *(green) 4& *(green) 5\\
\none & \none & *(cyan) 0 & *(cyan) 1 & *(green) 5
\end{ytableau}
\ytableausetup{nosmalltableaux}
    \caption{For $\lambda = (2,1)$ and $\mu = (5, 4, 3)$, we have that $|\SDT{(\mu, w(\lambda, 3, 3))}| = 4$.}
    \label{fig:trapezoid}
\end{figure}

\end{remark}

\bibliographystyle{plain}
\bibliography{references}

\end{document}